\documentclass[reqno, 12pt]{amsart}
\pdfoutput=1
\makeatletter
\let\origsection=\section \def\section{\@ifstar{\origsection*}{\mysection}}
\def\mysection{\@startsection{section}{1}\z@{.7\linespacing\@plus\linespacing}{.5\linespacing}{\normalfont\scshape\centering\S}}
\makeatother

\usepackage{amsmath,amssymb,amsthm}
\usepackage{mathrsfs}
\usepackage{mathabx}\changenotsign
\usepackage{dsfont}
\usepackage{bbm}

\usepackage{xcolor}
\usepackage[backref=section]{hyperref}
\usepackage[ocgcolorlinks]{ocgx2}
\hypersetup{
	colorlinks=true,
	linkcolor={red!60!black},
	citecolor={green!60!black},
	urlcolor={blue!60!black},
}

\usepackage[open,openlevel=2,atend]{bookmark}

\usepackage[abbrev,msc-links,backrefs]{amsrefs}
\usepackage{doi}

\renewcommand{\PrintDOI}[1]{\doi{#1}}

\usepackage[T1]{fontenc}
\usepackage{lmodern}
\usepackage[babel]{microtype}
\usepackage[english]{babel}

\usepackage{geometry}
\numberwithin{equation}{section}
\numberwithin{figure}{section}

\usepackage{enumitem}

\let\polishlcross=\l
\def\l{\ifmmode\ell\else\polishlcross\fi}

\let\emptyset=\varnothing
\let\setminus=\smallsetminus

\makeatletter
\def\moverlay{\mathpalette\mov@rlay}
\def\mov@rlay#1#2{\leavevmode\vtop{   \baselineskip\z@skip \lineskiplimit-\maxdimen
		\ialign{\hfil$\m@th#1##$\hfil\cr#2\crcr}}}
\newcommand{\charfusion}[3][\mathord]{
	#1{\ifx#1\mathop\vphantom{#2}\fi
		\mathpalette\mov@rlay{#2\cr#3}
	}
	\ifx#1\mathop\expandafter\displaylimits\fi}
\makeatother

\newcommand{\dcup}{\charfusion[\mathbin]{\cup}{\cdot}}

\DeclareFontFamily{U}  {MnSymbolC}{}
\DeclareSymbolFont{MnSyC}         {U}  {MnSymbolC}{m}{n}
\DeclareFontShape{U}{MnSymbolC}{m}{n}{
	<-6>  MnSymbolC5
	<6-7>  MnSymbolC6
	<7-8>  MnSymbolC7
	<8-9>  MnSymbolC8
	<9-10> MnSymbolC9
	<10-12> MnSymbolC10
	<12->   MnSymbolC12}{}
\DeclareMathSymbol{\powerset}{\mathord}{MnSyC}{180}

\usepackage{tikz}
\usetikzlibrary{calc,decorations.pathmorphing}
\usetikzlibrary{arrows,decorations.pathreplacing}
\pgfdeclarelayer{background}
\pgfdeclarelayer{foreground}
\pgfdeclarelayer{front}
\pgfsetlayers{background,main,foreground,front}

\usepackage{multicol}
\usepackage{subcaption}
\newcommand{\pedge}[9]{
	
	\ifx\relax#6\relax
	\def\qoffs{0pt}
	\else
	\def\qoffs{#6}
	\fi
	
	\def\phedge{
		($#1+#5!\qoffs!-90:#2-#5$) -- 
		($#2+#1!\qoffs!-90:#3-#1$) -- 
		($#3+#2!\qoffs!-90:#4-#2$) -- 
		($#4+#3!\qoffs!-90:#5-#3$) -- 
		($#5+#4!\qoffs!-90:#1-#4$) -- cycle}

	\coordinate (12) at ($#1!\qoffs!90:#2$);
	\coordinate (15) at ($#1!\qoffs!-90:#5$);
	\coordinate (23) at ($#2!\qoffs!90:#3$);
	\coordinate (21) at ($#2!\qoffs!-90:#1$);
	\coordinate (34) at ($#3!\qoffs!90:#4$);
	\coordinate (32) at ($#3!\qoffs!-90:#2$);
	\coordinate (45) at ($#4!\qoffs!90:#5$);
	\coordinate (43) at ($#4!\qoffs!-90:#3$);
	\coordinate (51) at ($#5!\qoffs!90:#1$);
	\coordinate (54) at ($#5!\qoffs!-90:#4$);

	\def\nphedge{
		(15) let \p1=($(15)-#1$), \p2=($(12)-#1$) in 
		arc[start angle={atan2(\y1,\x1)}, delta angle={atan2(\y2,\x2)-atan2(\y1,\x1)-360*(atan2(\y2,\x2)-atan2(\y1,\x1)>0)}, x radius=\qoffs, y radius=\qoffs] --
		(21) let \p1=($(21)-#2$), \p2=($(23)-#2$) in 
		arc[start angle={atan2(\y1,\x1)}, delta angle={atan2(\y2,\x2)-atan2(\y1,\x1)-360*(atan2(\y2,\x2)-atan2(\y1,\x1)>0)}, x radius=\qoffs, y radius=\qoffs] --
		(32) let \p1=($(32)-#3$), \p2=($(34)-#3$) in 
		arc[start angle={atan2(\y1,\x1)}, delta angle={atan2(\y2,\x2)-atan2(\y1,\x1)-360*(atan2(\y2,\x2)-atan2(\y1,\x1)>0)}, x radius=\qoffs, y radius=\qoffs] --
		(43) let \p1=($(43)-#4$), \p2=($(45)-#4$) in 
		arc[start angle={atan2(\y1,\x1)}, delta angle={atan2(\y2,\x2)-atan2(\y1,\x1)-360*(atan2(\y2,\x2)-atan2(\y1,\x1)>0)}, x radius=\qoffs, y radius=\qoffs] --
		(54) let \p1=($(54)-#5$), \p2=($(51)-#5$) in 
		arc[start angle={atan2(\y1,\x1)}, delta angle={atan2(\y2,\x2)-atan2(\y1,\x1)-360*(atan2(\y2,\x2)-atan2(\y1,\x1)>0)}, x radius=\qoffs, y radius=\qoffs] --
		cycle}

	\ifx\relax#7\relax
	\def\plwidth{1pt}
	\else
	\def\plwidth{#7}
	\fi
	
	\ifx\relax#9\relax
	\fill \nphedge;
	\else
	\fill[#9]\nphedge;
	\fi
	
	\ifx\relax#8\relax
	\draw[line width=\plwidth,rounded corners=\qoffs]\nphedge;
	\else
	\draw[line width=\plwidth,#8]\nphedge;
	\fi
}

\newcommand{\qedge}[7]{
	
	\ifx\relax#4\relax
	\def\qoffs{0pt}
	\else
	\def\qoffs{#4}
	\fi
	
	\def\qhedge{
		($#1+#3!\qoffs!-90:#2-#3$) --
		($#2+#1!\qoffs!-90:#3-#1$) --
		($#3+#2!\qoffs!-90:#1-#2$) -- cycle}

	\coordinate (12) at ($#1!\qoffs!90:#2$);
	\coordinate (13) at ($#1!\qoffs!-90:#3$);
	\coordinate (23) at ($#2!\qoffs!90:#3$);
	\coordinate (21) at ($#2!\qoffs!-90:#1$);
	\coordinate (31) at ($#3!\qoffs!90:#1$);
	\coordinate (32) at ($#3!\qoffs!-90:#2$);
	
	\def\nqhedge{
		(13) let \p1=($(13)-#1$), \p2=($(12)-#1$) in
		arc[start angle={atan2(\y1,\x1)}, delta angle={atan2(\y2,\x2)-atan2(\y1,\x1)-360*(atan2(\y2,\x2)-atan2(\y1,\x1)>0)}, x radius=\qoffs, y radius=\qoffs] --
		(21) let \p1=($(21)-#2$), \p2=($(23)-#2$) in
		arc[start angle={atan2(\y1,\x1)}, delta angle={atan2(\y2,\x2)-atan2(\y1,\x1)-360*(atan2(\y2,\x2)-atan2(\y1,\x1)>0)}, x radius=\qoffs, y radius=\qoffs] --
		(32) let \p1=($(32)-#3$), \p2=($(31)-#3$) in
		arc[start angle={atan2(\y1,\x1)}, delta angle={atan2(\y2,\x2)-atan2(\y1,\x1)-360*(atan2(\y2,\x2)-atan2(\y1,\x1)>0)}, x radius=\qoffs, y radius=\qoffs] --
		cycle}
	
	\ifx\relax#5\relax
	\def\qlwidth{1pt}
	\else
	\def\qlwidth{#5}
	\fi
	
	\ifx\relax#7\relax
	\fill \nqhedge;
	\else
	\fill[#7]\nqhedge;
	\fi
	
	\ifx\relax#6\relax
	\draw[line width=\qlwidth,rounded corners=\qoffs]\nqhedge;
	\else
	\draw[line width=\qlwidth,#6]\nqhedge;
	\fi
}

\newcommand{\redge}[8]{
	
	\ifx\relax#5\relax
	\def\qoffs{0pt}
	\else
	\def\qoffs{#5}
	\fi
	
	\def\rhedge{
		($#1+#4!\qoffs!-90:#2-#4$) -- 
		($#2+#1!\qoffs!-90:#3-#1$) -- 
		($#3+#2!\qoffs!-90:#4-#2$) -- 
		($#4+#3!\qoffs!-90:#1-#3$) -- cycle}

	\coordinate (12) at ($#1!\qoffs!90:#2$);
	\coordinate (14) at ($#1!\qoffs!-90:#4$);
	\coordinate (23) at ($#2!\qoffs!90:#3$);
	\coordinate (21) at ($#2!\qoffs!-90:#1$);
	\coordinate (34) at ($#3!\qoffs!90:#4$);
	\coordinate (32) at ($#3!\qoffs!-90:#2$);
	\coordinate (41) at ($#4!\qoffs!90:#1$);
	\coordinate (43) at ($#4!\qoffs!-90:#3$);
	
	\def\nrhedge{
		(14) let \p1=($(14)-#1$), \p2=($(12)-#1$) in 
		arc[start angle={atan2(\y1,\x1)}, delta angle={atan2(\y2,\x2)-atan2(\y1,\x1)-360*(atan2(\y2,\x2)-atan2(\y1,\x1)>0)}, x radius=\qoffs, y radius=\qoffs] --
		(21) let \p1=($(21)-#2$), \p2=($(23)-#2$) in 
		arc[start angle={atan2(\y1,\x1)}, delta angle={atan2(\y2,\x2)-atan2(\y1,\x1)-360*(atan2(\y2,\x2)-atan2(\y1,\x1)>0)}, x radius=\qoffs, y radius=\qoffs] --
		(32) let \p1=($(32)-#3$), \p2=($(34)-#3$) in 
		arc[start angle={atan2(\y1,\x1)}, delta angle={atan2(\y2,\x2)-atan2(\y1,\x1)-360*(atan2(\y2,\x2)-atan2(\y1,\x1)>0)}, x radius=\qoffs, y radius=\qoffs] --
		(43) let \p1=($(43)-#4$), \p2=($(41)-#4$) in 
		arc[start angle={atan2(\y1,\x1)}, delta angle={atan2(\y2,\x2)-atan2(\y1,\x1)-360*(atan2(\y2,\x2)-atan2(\y1,\x1)>0)}, x radius=\qoffs, y radius=\qoffs] --
		cycle}
	
	\ifx\relax#6\relax
	\def\rlwidth{1pt}
	\else
	\def\rlwidth{#6}
	\fi
	
	\ifx\relax#8\relax
	\fill \nrhedge;
	\else
	\fill[#8]\nrhedge;
	\fi
	
	\ifx\relax#7\relax
	\draw[line width=\rlwidth,rounded corners=\qoffs]\nrhedge;
	\else
	\draw[line width=\rlwidth,#7]\nrhedge;
	\fi
}

\let\epsilon=\varepsilon

\let\rho=\varrho
\let\theta=\vartheta

\newtheoremstyle{note}  {4pt}  {4pt}  {\sl}  {}  {\bfseries}  {.}  {.5em}          {}
\newtheoremstyle{introthms}  {3pt}  {3pt}  {\itshape}  {}  {\bfseries}  {.}  {.5em}          {\thmnote{#3}}
\newtheoremstyle{remark}  {2pt}  {2pt}  {\rm}  {}  {\bfseries}  {.}  {.3em}          {}

\theoremstyle{plain}
\newtheorem{theorem}{Theorem}[section]
\newtheorem{lemma}[theorem]{Lemma}
\newtheorem{prop}[theorem]{Proposition}
\newtheorem{constr}{Construction}

\newtheorem{cor}[theorem]{Corollary}

\theoremstyle{note}

\theoremstyle{remark}

\newtheorem{question}[theorem]{Question}

\newtheorem{problem}[theorem]{Problem}

\usepackage{lineno}
\newcommand*\patchAmsMathEnvironmentForLineno[1]{
	\expandafter\let\csname old#1\expandafter\endcsname\csname #1\endcsname
	\expandafter\let\csname oldend#1\expandafter\endcsname\csname end#1\endcsname
	\renewenvironment{#1}
	{\linenomath\csname old#1\endcsname}
	{\csname oldend#1\endcsname\endlinenomath}}
\newcommand*\patchBothAmsMathEnvironmentsForLineno[1]{
	\patchAmsMathEnvironmentForLineno{#1}
	\patchAmsMathEnvironmentForLineno{#1*}}
\AtBeginDocument{
	\patchBothAmsMathEnvironmentsForLineno{equation}
	\patchBothAmsMathEnvironmentsForLineno{align}
	\patchBothAmsMathEnvironmentsForLineno{flalign}
	\patchBothAmsMathEnvironmentsForLineno{alignat}
	\patchBothAmsMathEnvironmentsForLineno{gather}
	\patchBothAmsMathEnvironmentsForLineno{multline}
}

\def\ex{\text{\rm ex}}

\usepackage{scalerel}

\makeatletter
\newcommand{\overrighharpoonup}[1]{\ThisStyle{%
		\vbox {\m@th\ialign{##\crcr
				\rightharpoonupfill \crcr
				\noalign{\kern-\p@\nointerlineskip}
				$\hfil\SavedStyle#1\hfil$\crcr}}}}

\def\rightharpoonupfill{%
	$\SavedStyle\m@th\mkern+0.8mu\cleaders\hbox{$\shortbar\mkern-4mu$}\hfill\rightharpoonuptip\mkern+0.8mu$}

\def\rightharpoonuptip{%
	\raisebox{\z@}[2pt][1pt]{\scalebox{0.55}{$\SavedStyle\rightharpoonup$}}}

\def\shortbar{%
	\smash{\scalebox{0.55}{$\SavedStyle\relbar$}}}
\makeatother

\makeatletter
\newcommand{\overlefharpoonup}[1]{\ThisStyle{%
		\vbox {\m@th\ialign{##\crcr
				\leftharpoonupfill \crcr
				\noalign{\kern-\p@\nointerlineskip}
				$\hfil\SavedStyle#1\hfil$\crcr}}}}

\def\leftharpoonupfill{%
	$\SavedStyle\m@th\mkern+0.8mu\cleaders\hbox{$\shortbar\mkern-4mu$}\hfill\leftharpoonuptip\mkern+0.8mu$}

\def\leftharpoonuptip{%
	\raisebox{\z@}[2pt][1pt]{\scalebox{0.55}{$\SavedStyle\leftharpoonup$}}}

\makeatletter
\newsavebox\myboxA
\newsavebox\myboxB
\newlength\mylenA

\newcommand*\xoverline[2][0.75]{%
	\sbox{\myboxA}{$\m@th#2$}%
	\setbox\myboxB\null
	\ht\myboxB=\ht\myboxA%
	\dp\myboxB=\dp\myboxA%
	\wd\myboxB=#1\wd\myboxA
	\sbox\myboxB{$\m@th\overline{\copy\myboxB}$}
	\setlength\mylenA{\the\wd\myboxA}
	\addtolength\mylenA{-\the\wd\myboxB}%
	\ifdim\wd\myboxB<\wd\myboxA%
	\rlap{\hskip 0.5\mylenA\usebox\myboxB}{\usebox\myboxA}%
	\else
	\hskip -0.5\mylenA\rlap{\usebox\myboxA}{\hskip 0.5\mylenA\usebox\myboxB}%
	\fi}
\makeatother

\begin{document}
	
	\title[Separating hypergraph Tur\'{a}n densities]
	{Separating hypergraph Tur\'{a}n densities}

\author[Hong Liu]{Hong Liu}
	\address{Extremal Combinatorics and Probability Group, Institute for Basic Science, Daejeon, South Korea}
    \email{\{hongliu, schuelke\}@ibs.re.kr}

\author[Bjarne Sch\"{u}lke]{Bjarne Sch\"{u}lke}
 
 \author[Shuaichao Wang]{Shuaichao Wang}
	\address{Center for Combinatorics and LPMC, Nankai University, Tianjin, China, and Extremal Combinatorics and Probability Group, Institute for Basic Science, Daejeon, South Korea}
	\email{wsc17746316863@163.com}

\author[Haotian Yang]{Haotian Yang}
\address{Mathematics Department, California Institute of Technology, Pasadena, USA and Extremal Combinatorics and Probability Group, Institute for Basic Science, Daejeon, South Korea}
\email{hyang3@caltech.edu}

 \author[Yixiao Zhang]{Yixiao Zhang}
	\address{Center for Discrete Mathematics, Fuzhou University, Fujian, China, and Extremal Combinatorics and Probability Group, Institute for Basic Science, Daejeon, South Korea}
	\email{fzuzyx@gmail.com}


\begin{abstract}
Determining the Tur\'an densities of hypergraphs is a notoriously difficult problem at the core of combinatorics.
Although Tur\'an posed this problem in 1941,~$\pi(K_{\ell}^{(k)})$ remains unknown for all~$\ell>k\geq 3$.
Prior to this work, it was not even known whether~$\pi(K_{\ell}^{(k)})<\pi(K_{\ell+1}^{(k)})$ holds for general~$\ell$ and~$k$, and the best-known bounds on~$\pi(K_{\ell}^{(k)})$ are far from implying anything close to this.
We prove that~$\pi(K_{\ell}^{(k)})<\pi(K_{\ell+1}^{(k)})$, for all~$\ell>k\geq 3$, and provide a general criterion to distinguish the Tur\'an densities of two hypergraphs.
As a corollary, we obtain that~$\pi(K_{k+1}^{(k)})<\pi(K_{k+2}^{(k)-})$, for all~$k\geq 3$.
For~$k=3$, this was previously proved by Markstr\"om, answering a question by Erd\H{o}s.

\end{abstract}

\maketitle
	
\section{Introduction}\label{SEC:Introduction}

The Tur\'an problem is one of the most important themes in extremal combinatorics.
Roughly speaking, it asks for the threshold of the edge density above which every large hypergraph must contain a copy of a fixed hypergraph~$F$.
A~$k$-uniform hypergraph (or~$k$-graph)~$H=(V,E)$ consists of a vertex set~$V$ and an edge set~$E\subseteq V^{(k)}$.
For~$k$-graphs~$F$ and~$H$, we say that~$H$ is~$F$-free if~$H$ does not contain a copy of~$F$. 
Given~$n\in\mathds{N}$, the Tur\'{a}n number~$\mathrm{ex}(n,F)$ of~$F$ is the maximum number of edges in an~$F$-free~$k$-graph on~$n$ vertices. 
The Tur\'{a}n density of~$F$ is defined as 
\begin{align*}
    \pi(F) = \lim_{n \to \infty}\frac{\mathrm{ex}(n,F)}{\binom{n}{k}}.
\end{align*}
This limit can be shown to exist by a simple monotonicity argument~\cite{KNS64}.
For graphs, i.e.,~$k=2$, the Tur\'{a}n density is completely understood by the results of Tur\'{a}n~\cite{Turan41}, Erd\H{o}s and Stone~\cite{ES46}, and Erd\H{o}s and Simonovits~\cite{ES66}.
The latter result states that for every graph~$F$, we have~$\pi(F) = \frac{\chi(F)-2}{\chi(F)-1}$, where~$\chi(F)$ is the chromatic number of~$F$.

For~$k \geq 3$, determining the Tur\'{a}n density is notoriously difficult in general and very few results are known. 
Let~$K_{\ell}^{(k)}$ denote the complete~$k$-graph on~$\ell$ vertices and let~$K_{\ell}^{(k)-}$ denote~$K_{\ell}^{(k)}$ minus one edge. 
The problem of determining~$\pi(K_{\ell}^{(k)})$ was raised by Tur\'{a}n~\cite{Turan41} back in the 1940s and is still widely open for all~$\ell > k \geq 3$ despite receiving a great deal of attention from various researchers over the years. 
Erd\H{o}s offered {\$500} for determining any~$\pi(K_{\ell}^{(k)})$ with~$\ell > k\geq 3$ and~{\$1000} for determining all~$\pi(K_{\ell}^{(k)})$ with~$\ell >  k\geq 3$.
Beyond basic results like supersaturation, there is no general criterion known that forces two hypergraphs to have the same Tur\'an density or that guarantees their Tur\'an densities to be distinct.
Before this work, it was not even known whether~$\pi(K_{\ell}^{(k)})<\pi(K_{\ell+1}^{(k)})$ holds for general~$k$ and~$\ell$.
In fact, the best known general upper and lower bounds on the Tur\'an densities of cliques are far from implying anything close to this and even vast improvements on these bounds are unlikely to yield such a result.
Here we show that in any uniformity, the Tur\'an densities of cliques of different sizes are separated.

\begin{theorem}\label{COR:gap-complete-or-minus}
    For~$\ell > k \geq 3$, we have~$\pi(K_{\ell}^{(k)}) < \pi(K_{\ell+1}^{(k)})$ and~$\pi(K_{\ell}^{(k)-}) < \pi(K_{\ell+1}^{(k)-})$.
\end{theorem}

A closely related question is the following. Dirac and Erd\H{o}s observed that any graph on~$n$ vertices with~$\mathrm{ex}(n,K_{\ell})$+1 edges contains not only a copy of~$K_{\ell}$ but also a copy of~$K_{\ell+1}^{\,-}$. 
Motivated by this, Erd\H{o}s~\cite{Erdos94} posed the following question for~$3$-graphs (see also~\cite{CG98,Chungpage}).

\begin{question}[Erd\H{o}s~\cite{Erdos94}]\label{CONJ:erdos-problem}
    Is it true that~$\ex(n,K_{\ell+1}^{(3)-})=\ex(n,K_{\ell}^{(3)})$ holds for every~$\ell\geq4$?
\end{question}

In this formulation, we can actually argue quickly that the answer is negative (see Proposition~\ref{prop:quick}).
For~$\ell=4$, Markstr\"{o}m~\cite{Mark14} proved that the answer to Question~\ref{CONJ:erdos-problem} is negative in a strong sense by showing that~$\pi(K_4^{(3)})<\pi(K_5^{(3)-})$.
For his proof he constructed a~$K_5^{(3)-}$-free~$3$-graph that provided a lower bound for~$\pi(K_5^{(3)-})$ that is larger than the upper bounds for~$\pi(K_4^{(3)})$ given by computer-aided flag algebra calculations in~\cites{Ra10,RB12}.
In this paper, we generalise Markstr\"om's result to~$k$-graphs, that is, we separate the Tur\'{a}n densities of~$K_{k+1}^{(k)}$ and~$K_{k+2}^{(k)-}$ for all~$k\geq 3$.
We remark that all our proofs are flag-algebra-free.
\begin{theorem}\label{THM:gapof-Kk+1-Kk+2-}
    For every integer~$k \geq 3$, we have~$\pi(K_{k+1}^{(k)}) < \pi(K_{k+2}^{(k)-})$.
\end{theorem}

For more results on hypergraph Tur\'{a}n problems, we refer the reader to the surveys~\cites{BCL22,Furedi91,Ke11,Si95}. 
Given that it seems completely out of reach to determine the Tur\'an densities for all hypergraphs, there has been a considerable effort to describe the distribution of the set of all Tur\'{a}n densities of~$k$-graphs~(see, e.g.,~\cites{Pik14,CS24,LP23}).
In~\cite{CS24}, Conlon and Sch\"ulke show that the set~$\{\pi(F):F\text{ is a }k\text{-graph}\}$ has an accumulation point for every~$k\geq 3$.
As part of their argument they show that the Tur\'an densities of certain hypergraphs are distinct, without having explicit bounds for them.

In fact, both Theorem~\ref{COR:gap-complete-or-minus} and Theorem~\ref{THM:gapof-Kk+1-Kk+2-} are corollaries of a general criterion that we establish to guarantee that the Tur\'an densities of two hypergraphs are distinct.

\begin{theorem}\label{THM:general-result}
    Let~$m > k \geq 3$, let~$F$ be a~$k$-graph on~$m$ vertices, and let~$F'\subseteq F$.
    Suppose that one of the following conditions holds:
    \begin{enumerate}
        \item $\mathrm{ex}(m,F') + \prod_{i=0}^{k-1} \left\lfloor \frac{m+i}{k} \right\rfloor < e(F)$; \label{EQ:Statement-1}
        \item or for every $k$-partition~$V(F) = V_1 \dcup \dots \dcup V_k$ such that every edge of~$F$ intersects~$V_1$, there is some~$2 \leq j \leq k$ such that~$F'$ is contained in~$F-V_j$.\label{EQ:Statement-2}
    \end{enumerate}
    Then~$\pi(F')<\pi(F)$.
\end{theorem}

To prove the $k=3$ case of Theorem~\ref{THM:gapof-Kk+1-Kk+2-}, we present a new construction of a~$K_5^{(3)-}$-free~$3$-graph, improving the best known lower bound $\pi(K_5^{(3)-}) \geq 0.58656$ due to
Balogh, Clemen, and Lidick\'{y}~\cite{BCL22}. 

\begin{theorem}\label{THM:lower-bound-K53-}
We have~$\pi(K_5^{(3)-}) \geq \frac{31097 + 277 \sqrt{277}}{59248} = 0.602673\ldots.$
\end{theorem}

Theorem~\ref{THM:general-result} provides a criterion to separate many pairs of Tur\'{a}n densities without determining their values explicitly.
For another interesting example, let~$H_{t}^{(k)}$ denote the~$k$-graph with~$k+1$ vertices and~$t \leq k+1$ edges (for fixed~$k$ and~$t$, all such~$k$-graphs are isomorphic and sometimes they are refered to as daisies). 
For its Tur\'{a}n density, the well-known upper bound is~$\pi(H_{t}^{(k)}) \leq \frac{t-2}{k}$~(see~\cite{MT21}).
Very rencently, Sidorenko~\cite{Sid24} proved that~$\pi(H_{t}^{(k)}) \geq \left(C_t -o(1)\right) k^{-(1+\frac{1}{t-2})}$ as~$k \to \infty$.
For the case~$t=k+1$, Pikhurko~\cite{Pik24} showed that~$\pi(K_{k+1}^{(k)}) \geq 1-\frac{6.239}{k+1}$ (and~$\pi(K_{k+1}^{(k)}) \geq1-\frac{4.911}{k+1}$ for large~$k$).
These bounds imply that~$\pi(H_t^{(k)})<\pi(H_{t'}^{(k)})$ for~$t'$ sufficiently large compared to~$t$.
Our methods provide a much finer separation of these Tur\'{a}n densities.

\begin{cor}\label{COR:gap-k+1vertex-t-t+2}
    For~$1 \leq t \leq k-1$ and~$k\geq2$, we have~$\pi(H_{t}^{(k)}) < \pi(H_{t+2}^{(k)})$ .
\end{cor}

This paper is organized as follows. 
In Section~\ref{SEC:general-proof-seperating}, we prove Theorem~\ref{THM:general-result} and use it to prove Theorem~\ref{COR:gap-complete-or-minus}, Theorem~\ref{THM:gapof-Kk+1-Kk+2-} for~$k \geq 4$, and Corollary~\ref{COR:gap-k+1vertex-t-t+2}. 
In Section~\ref{SEC:Lower-bound-K53-}, we provide a construction of a~$K_5^{(3)-}$-free~$3$-graph to prove Theorem~\ref{THM:gapof-Kk+1-Kk+2-} for~$k=3$. 
Section~\ref{SEC:CONCLUDING-Remarks} contains some concluding remarks and open problems.

\subsection*{Notation}
For a~$k$-graph~$H$, we write~$e(H)=\vert E(H)\vert$ and~$v(H)=\vert V(H)\vert$.
For a set~$S\subseteq V(H)$,~$H[S]$ is the~$k$-graph with vertex set~$S$ and edge set~$\{e\in E(H):e\subseteq S\}$. We further set~$H-S=H[V\setminus S]$.

\section{Proof of Main result}\label{SEC:general-proof-seperating}

Before proving the main result, we briefly return to the question asked by Erd\H{o}s (Question~\ref{CONJ:erdos-problem}) and give a short argument that the answer is negative.
It already indicates the strategy we will use later.
In contrast to Markstr\"oms original proof, our proof does not rely on any computer-assisted calculations.

\begin{prop}\label{prop:quick}
    For all~$n\in\mathds{N}$, we have $$\ex(n,K_{5}^{(3)-})\geq\ex(n,K_{4}^{(3)})+\left\lfloor\frac{n}{5}\right\rfloor\,.$$
\end{prop}

\begin{proof}
    Let~$H$ be a~$K_{4}^{(3)}$-free~$3$-graph on~$n$ vertices with~$\ex(n,K_{4}^{(3)})$ edges.
    Now let~$D$ be an almost perfect~$5$-uniform matching on~$V(H)$, i.e., let~$D\subseteq V(H)^{(5)}$ with~$\vert D\vert=\left\lfloor\frac{n}{5}\right\rfloor$ such that the sets in~$D$ are pairwise disjoint.
    Observe that since $H$ must miss at least one edge on every~$4$-set of vertices,~$H$ misses at least three edges on every~$f\in D$ (also note that strictly speaking, this already answers the question).
    Let~$H^+$ be the~$3$-graph obtained from~$H$ by adding one of those missing edges on every~$f\in D$.
    Then~$e(H^+)\geq e(H)+\left\lfloor\frac{n}{5}\right\rfloor$.
    But note that~$H^+$ is still~$K_{5}^{(3)-}$-free.
    Indeed assume that~$H^+$ contains a copy of~$K_{5}^{(3)-}$ on~$\{v_1,\dots,v_5\}$.
    As mentioned before,~$H$ has three missing edges on~$\{v_1,\dots,v_5\}$.
    Thus,~$H^+$ must contain two newly added edges on~$\{v_1,\dots,v_5\}$.
    These two edges must intersect.
    This is a contradiction since these edges must be contained in distinct elements of the matching~$D$.
\end{proof}
Clearly, this kind of argument can be adapted for other settings, for instance by using (partial) designs, but we omit a further discussion.

For the proof of Theorem~\ref{THM:general-result}, we need the following lemma.

\begin{lemma}\label{LEM:Regularity-patition}
    Given~$\varepsilon>0$, an integer~$k \geq 3$, and a~$k$-graph~$F$, there are~$T_0, N_0 \in \mathbb{N}$ such that for every~$F$-free~$k$-graph~$H$ on~$n \geq N_0$ vertices, the following holds.
    There exist an integer~$t_0 \leq T_0$ and~$k$ pairwise disjoint sets~$U_1, \ldots,U_k\subseteq V(H)$ such that
    \begin{itemize}
        \item $|U_i| = \frac{n}{t_0}$, for every~$i \in [k]$ and
        \item $\vert E_c\vert\leq\left(\pi(F)+ \epsilon \right)\big(\frac{n}{t_0}\big)^k$, where~$E_c=\{e\in E(H):\vert e\cap U_i\vert=1\text{ for all }i\in[k]\}$. 
    \end{itemize}
\end{lemma}

\begin{proof}
    Choose~$T_0$ and~$N_0$ such that~$T_0^{-1},N_0^{-1}\ll\varepsilon,k^{-1}$ and let~$n\geq N_0$.
    Note that, in particular, we have~$\left| \mathrm{ex}(n,F)/ \binom{n}{k} - \pi(F)\right| < \epsilon$. 
    Let~$H$ be an~$F$-free~$k$-graph on~$n$ vertices and note that~$ e(H) \leq \mathrm{ex}(n,F) \leq (\pi(F) + \epsilon) \binom{n}{k}$.
    Randomly choose~$k$ pairwise disjoint sets~$U_1,\ldots,U_k\subseteq V(H)$ with~$|U_i| = \frac{n}{t_0}$, for every~$i \in [k]$. 
    The total number of choices is
    \begin{align*}
        \binom{n}{\frac{n}{t_0}} \binom{n-\frac{n}{t_0}}{\frac{n}{t_0}} \cdots \binom{n-(k-1)\frac{n}{t_0}}{\frac{n}{t_0}} = \frac{n!}{((\frac{n}{t_0})!)^k (n-k\frac{n}{t_0})!}.
    \end{align*}
    Let~$E_c=\{e\in E(H):\vert e\cap U_i\vert=1\text{ for all }i\in[k]\}$. 
    For every~$e_0\in E(H)$, the number of choices such that~$e_0 \in E_c$ is
    \begin{align*}
        k!\binom{n-k}{\frac{n}{t_0}-1} \binom{n-k-\big(\frac{n}{t_0}-1\big)}{\frac{n}{t_0} -1} \cdots \binom{n-(k-1)\frac{n}{t_0}-1}{\frac{n}{t_0} - 1} = \frac{k!(n-k)!}{((\frac{n}{t_0}-1)!)^k  (n-k\frac{n}{t_0})!}.
    \end{align*}
    Hence, for every edge~$e_0\in E(H)$,~we have 
    \begin{align*}
        \mathbb{P}\left(e_0 \in E_c \right) = \frac{k!(n-k)! \big(\frac{n}{t_0}\big)^k}{n!}.
    \end{align*}
    This implies
    \begin{align*}
        \mathbb{E}(|E_c|) = e(H) \frac{k! (n-k)! \big(\frac{n}{t_0}\big)^k}{n!} \leq \left(\pi(F) + \epsilon \right) \Big(\frac{n}{t_0}\Big)^k. 
    \end{align*}
    Therefore, there indeed exists some choice of~$U_1,\ldots, U_k$ as claimed.
 \end{proof}

Now we can prove our main theorem.
\begin{proof}[Proof of Theorem~\ref{THM:general-result}]
    Let~$m > k \geq 3$, let~$F$ be a~$k$-graph on~$m$ vertices, and let~$F'\subseteq F$ be such that one of the conditions in Theorem~\ref{THM:general-result} is satisfied. 
    Choose~$\epsilon,\delta>0$ and~$N_0,T_0\in\mathds{N}$ such that~$N_0^{-1},T_0^{-1}\ll\varepsilon,\delta\ll m^{-1},k^{-1}$.
    In particular, we have~$\epsilon,\delta \ll \binom{m-1}{k-1}^{-1}-\binom{m}{k-1}^{-1}$ and the conclusion of Lemma~\ref{LEM:Regularity-patition} holds for~$F'$,~$\varepsilon$,~$N_0$, and~$T_0$. 
    Let~$H$ be an~$F'$-free~$k$-graph on~$n \geq N_0$ vertices with~$\mathrm{ex}(n,F')$ edges. 
    Since the number of vertices of~$F'$ is at most~$m$, by a known upper bound for~$\pi(K_m^{(k)})$ (see~\cite{DC83}), we have
   
        $$\pi(F') \leq \pi(K_m^{(k)}) \leq 1-\binom{m-1}{k-1}^{-1}.$$
    
   Then by Lemma~\ref{LEM:Regularity-patition}, there exist an integer~$t_0\leq T_0$ and~$k$ pairwise disjoint sets~$U_1, \ldots U_k\subseteq V(H)$ such that for every~$i \in [k]$,~$|U_i| = \frac{n}{t_0}$ and the number of edges intersecting every part in exactly one vertex is at most
     \begin{align}\label{EQ:crossing-upper-bound}
         \left(1-\binom{m-1}{k-1}^{-1}+ \epsilon \right)\left(\frac{n}{t_0}\right)^k \leq \left(1-\binom{m}{k-1}^{-1} \right)\left(\frac{n}{t_0}\right)^k.
     \end{align} 
    
   First, if~$F'$ satisfies Condition~\eqref{EQ:Statement-1} in Theorem~\ref{THM:general-result}, let~$H_1$ be the~$k$-graph obtained from~$H$ by adding all~$k$-sets which intersect every~$U_i$ in exactly one vertex as edges.
   Note that~$e(H_1)\geq e(H)+\binom{m}{k-1}^{-1}\big(\frac{n}{t_0}\big)^k$.
   We claim that~$H_1$ is~$F$-free. 
   Assume, for the sake of a contradiction, that it contains a copy of~$F$ (we will simply refer to this copy as~$F$). 
   Note that this~$F$ contains at most~$\prod_{i=0}^{k-1} \left\lfloor \frac{m+i}{k} \right\rfloor$ newly added edges (i.e., elements of~$E(H_1)\setminus E(H)$).
   Thus~$H[V(F)]$ is a~$k$-graph on~$m$ vertices with at least~$e(F)-\prod_{i=0}^{k-1}\left\lfloor \frac{m+i}{k} \right\rfloor$ edges.
   Due to Condition~\eqref{EQ:Statement-1}, this implies that~$F'\subseteq H[V(F)]\subseteq H$, a contradiction.
   In summary, we have shown that for every large enough~$n\in\mathds{N}$,~$\ex(n,F')+\binom{m}{k-1}^{-1}\big(\frac{n}{t_0}\big)^k\leq\ex(n,F)$.
   This entails~$\pi(F')<\pi(F)$.

   Now assume that~$F'$ satisfies Condition~\eqref{EQ:Statement-2} in Theorem~\ref{THM:general-result}.
   For every~$2 \leq j \leq k$, we choose a random subset~$W_j \subseteq U_j$ of size~$\delta n/t_0$, where we select each set independently. 
   Similar to the proof of Lemma~\ref{LEM:Regularity-patition}, by~\eqref{EQ:crossing-upper-bound} there exists a choice such that the number of edges intersecting~$U_1$ and every~$W_j$~($2 \leq j \leq k$) in exactly one vertex is at most
   \begin{align}\label{EQ:crossing-upper-bound-delta}
       \left(1-\binom{m}{k-1}^{-1} \right)\left( \frac{\delta n}{t_0}\right)^{k-1} \times \frac{n}{t_0} = \delta^{k-1} \left(1-\binom{m}{k-1}^{-1} \right)\left(\frac{n}{t_0}\right)^k.
   \end{align}
   Let~$H_2$ be the~$k$-graph obtained from~$H$ by deleting all edges which are completely contained in~$\bigcup_{j=2}^{k} W_j$ and adding all edges which intersect~$U_1$ and every~$W_j$~($2 \leq j \leq k$) in exactly one vertex.
   Then $$e(H_2)\geq e(H)-\binom{\frac{(k-1)\delta n}{t_0}}{k}+\delta^{k-1}\binom{m}{k-1}^{-1}\big(\frac{n}{t_0}\big)^k\geq e(H)+\frac{1}{2}\delta^{k-1}\binom{m}{k-1}^{-1}\Big(\frac{n}{t_0}\Big)^k\,.$$
   We claim that~$H_2$ is~$F$-free. 
   Assume, for the sake of a contradiction, that it contains a copy of~$F$ (which we will simply refer to as~$F$).
   Let~$W_1 = V(H) \setminus \bigcup_{j=2}^k W_j$ and for every~$1 \leq j \leq k$, set~$V_j:= V(F) \cap W_j$. 
   Since~$H$ is~$F'$-free and thereby~$F$-free, one of the edges of~$F$ must be an edge in~$E(H_2)\setminus E(H)$.
   Thus,~$F$ must contain at least one vertex in each of~$U_1$, and~$W_j$ ($2\leq j\leq k$), implying that for every~$1 \leq j \leq k$, we have~$\vert V_j\vert \geq 1$. 
   Since no edge of~$H_2$ is completely contained in~$\bigcup_{j=2}^{k} W_j$, $V(F)=V_1\dcup\dots\dcup V_k$ is a~$k$-partition of~$V(F)$ such that every edge of~$F$ contains at least one vertex in~$V_1$. 
   By Condition~\eqref{EQ:Statement-2}, there is some $2 \leq j \leq k$ such that~$F'$ is also a subgraph of~$F-V_j$. 
   But note that then~$F'\subseteq F-V_j\subseteq H_2-W_j$.
   This yields~$F'\subseteq H$ because~$H_2-W_j$ does not contain any of the newly added edges, a contradiction. 
   Hence,~$H_2$ is indeed~$F$-free, and we have shown that for every large enough~$n\in\mathds{N}$,~$\ex(n,F')+\frac{1}{2}\delta^{k-1}\binom{m}{k-1}^{-1}\Big(\frac{n}{t_0}\Big)^k\leq\ex(n,F)$.
   This entails that~$\pi(F') < \pi(F)$.
\end{proof}

 \begin{proof}[Proof of Theorem~\ref{COR:gap-complete-or-minus}]
        Both statements follow from Condition~\eqref{EQ:Statement-2} in Theorem~\ref{THM:general-result}.
        To see that~$K_{\ell}^{(k)}$ and~$K_{\ell+1}^{(k)}$ satisfy this condition, let~$V(K_{\ell+1}^{(k)})=V_1 \dcup \dots \dcup V_k$ be a~$k$-partition such that every edge of~$K_{\ell+1}^{(k)}$ contains at least one vertex in~$V_1$.
        We may assume that~$V_j\neq\emptyset$ for all~$1\leq j\leq k$.
        Note that then~$\vert V_j\vert=1$ for all~$2\leq j\leq k$ and so~$K_{\ell}^{(k)}\subseteq K_{\ell+1}^{(k)}-V_j$ for every~$2\leq j\leq k$.
        To see that~$K_{\ell}^{(k)-}$ and~$K_{\ell+1}^{(k)-}$ satisfy this condition, let~$V(K_{\ell+1}^{(k)-})=V_1 \dcup \dots \dcup V_k$ be a~$k$-partition such that every edge of~$K_{\ell+1}^{(k)-}$ contains at least one vertex in~$V_1$.
        Again, we may assume that~$V_j\neq\emptyset$ for all~$1\leq j\leq k$.
        Note that there must be some integer~$i$ with~$2\leq i\leq k$ such that~$\vert V_j\vert=1$ for all~$j\in[k]\setminus\{1,i\}$ (for otherwise there would be two~$k$-sets in~$V_2\dcup\dots\dcup V_k$ and one of these would be an edge not containing a vertex in~$V_1$).
        Hence, using that every~$\ell$-set of vertices of~$K_{\ell+1}^{(k)-}$ has at most one missing edge, we conclude that~$K_{\ell}^{(k)-}\subseteq K_{\ell+1}^{(k)}-V_j$ for every~$j\in[k]\setminus\{1,i\}$.
    \end{proof}
 
 \begin{proof}[Proof of Theorem~\ref{THM:gapof-Kk+1-Kk+2-} for~$k \geq 4$]
     For every integer~$k \geq 4$, we verify that~$K_{k+1}^{(k)}$ and~$K_{k+2}^{(k)-}$ satisfy Condition~\eqref{EQ:Statement-2} in Theorem~\ref{THM:general-result}. 
     Let~$V(K_{k+2}^{(k)-})=V_1 \dcup \dots \dcup V_k$ be a~$k$-partition such that every edge of~$K_{k+2}^{(k)-}$ contains at least one vertex in~$V_1$.
     We may assume that~$|V_i| \geq 1$ holds, for every~$2 \leq i \leq k$, since if~$V_j=\emptyset$, we have~$K_{k+1}^{(k)}\subseteq K_{k+2}^{(k)-}=K_{k+2}^{(k)-}-V_j$. 
     Hence~$1\leq |V_1| \leq 3$ and $|V_i| \geq 1$ for $2 \leq i \leq k$. 
     
     If~$|V_1| = 1$, the number of~$k$-sets completely contained in~$V_2 \dcup \dots \dcup V_k$ is~$k+1$ and at most one of these does not form an edge of~$K_{k+2}^{(k)-}$.
     Hence, there is an edge of~$K_{k+2}^{(k)-}$ that does not intersect~$V_1$, a contradiction.
     If~$|V_1| =2$, without loss of generality, we have~$|V_1|=|V_2|=2$ and~$|V_i| =1$ for all~$3 \leq i \leq k$.
     Then the~$k$-set $V_2 \dcup \dots \dcup V_k$ must be the one missing edge of~$K_{k+2}^{(k)-}$. 
     Thus, all~$k$-subsets of~$V_1\dcup\dots\dcup V_{k-1}$ form edges, in other words~$K_{k+1}^{(k)}\subseteq K_{k+2}^{(k)-}-V_k$. 
     If~$|V_1|= 3 < k$ (and hence~$\vert V_i\vert=1$ for~$2\leq i\leq k$), then the~$k$-set that is the missing edge of~$K_{k+2}^{(k)-}$ must intersect~$V_j$, for some~$2 \leq j \leq k$.
     Then~$K_{k+1}^{(k)}\subseteq K_{k+2}^{(k)-}-V_j$.

     So~$K_{k+1}^{(k)}$ and~$K_{k+2}^{(k)-}$ indeed satisfy Condition~\eqref{EQ:Statement-2}, whence by Theorem~\ref{THM:general-result}  we obtain $\pi(K_{k+1}^{(k)}) < \pi(K_{k+2}^{(k)-})$.
 \end{proof}

 Next, we give an example using Condition~\eqref{EQ:Statement-1}, namely we prove Corollary~\ref{COR:gap-k+1vertex-t-t+2}.

 \begin{proof}[Proof of Corollary~\ref{COR:gap-k+1vertex-t-t+2}] 
     We have~$\mathrm{ex} (k+1,H_{t}^{(k)}) =t-1$ and~$\prod_{i=0}^{k-1} \left\lfloor \frac{k+1+i}{k} \right\rfloor =2$.
     Hence~$H_{t}^{(k)}$ and~$H_{t+2}^{(k)}$ satisfy Condition~\eqref{EQ:Statement-1} of Theorem~\ref{THM:general-result}, which entails~$\pi(H_{t}^{(k)}) < \pi(H_{t+2}^{(k)})$. 
 \end{proof}

\section{The lower bound on~\texorpdfstring{$\pi(K_5^{(3)-})$}{}}\label{SEC:Lower-bound-K53-}

Using a computer-aided proof based on Razborov's flag algebras, Baber~\cite{RB12} showed that~$\pi(K_4^{(3)}) \leq 0.5615$, improving on earlier bounds in~\cite{KNS64,dC88,CL99,Ra10}.
The best-known upper bound for~$\pi(K_4^{(3)})$ that does not use computer calculations is due to Chung and Lu~\cite{CL99}, who proved
\begin{align}\label{EQ:upper-bound-K43-noflag}
    \pi(K_4^{(3)}) \leq \frac{3+\sqrt{17}}{12} = 0.593592\ldots\,.
\end{align} 
Improving the bound given by Markstr\"{o}m in~\cite{Mark14}, Balogh, Clemen, and Lidick\'{y}~\cite{BCL22} presented a construction giving~$\pi(K_5^{(3)-}) \geq 0.58656$. 
However, this value is still smaller than the best computer-free upper bound on~$\pi(K_4^{(3)})$.
In this section, we prove Theorem~\ref{THM:lower-bound-K53-} using a new construction of a~$K_5^{(3)-}$-free~$3$-graph which has a larger density than the upper bound in~\eqref{EQ:upper-bound-K43-noflag}, thereby giving a computer-free proof of~$\pi(K_4^{(3)})<\pi(K_5^{(3)-})$.
This completes the proof of Theorem~\ref{THM:gapof-Kk+1-Kk+2-}.

First we present some known constructions that will be used later. 

\begin{constr}\label{CONSTRC:K43-}
Let~$S_6$ be the~$3$-graph with vertex set~$[6]$ and edge set 
$$E(S_6)=\{123,234,345,451,512,136,246,356,256,146\}.$$
Let $S_6^*$ be the iterated blow-up of $S_6$.
\end{constr}

This construction due to Frankl and F\"uredi~\cite{frankl1984exact} is the best known lower bound construction for~$\pi(K_4^{(3)-})$ and its edge density is~$2/7$. 
\begin{constr}\label{CONSTRC:bipartite-K43}
Let $A=\{a_1,a_2,...,a_n\}$ and $B=\{b_1,b_2,...,b_n\}$. Let $G$ be a $3$-graph on $A \dcup B$ whose edges are all triples of the form $\{a_i,b_j,a_k\}$ and $\{a_i,b_j,b_k\}$ where $i,j<k$.
\end{constr}

This construction is due to DeBiasio and Jiang~\cite{DJ:14}.
It is not difficult to check that this~$3$-graph~$G$ is~$K_4^{(3)}$-free and that the number of its edges is~$\left(\frac{2}{3}-o(1)\right)n^3$. 
Now let us start the proof of Theorem~\ref{THM:lower-bound-K53-}.

\begin{proof}[Proof of Theorem~\ref{THM:lower-bound-K53-}]

For the sake of simplicity, we ignore any rounding issues in this proof.
Let~$H_0=(V,E)$ be the~$3$-graph on~$n$ vertices defined as follows.
For some later specified reals~$x,y\in[0,1]$, the vertex set is~$V=V_1\dcup \dots\dcup V_6$, where~$|V_{3i-1}| =|V_{3i-2}|= x n$ and~$|V_{3i}|= y n$,~$i \in [2]$. 
The edge set of~$H_0$ is
\begin{align*}\label{EQ:type1-edge}
    E=&\Big\{abc \colon a \in  V_i, b \in V_j,c \in V_k \text{~and~} i j k \in \binom{[6]}{3}\setminus \{123,456,126,345\} \Big\}\\
    \cup&\big\{abc:ab\in V_{3i}^{(2)},c\in V_{3i-1}\cup V_{3i-2},i\in[2]\big\}\\
    \cup&\big\{abc:ab\in V_{3i-1}^{(2)}\cup V_{3i-2}^{(2)},c\in V_{3(3-i)},i\in[2]\big\}
\end{align*}
Now, place in each~$V_i$,~$i\in[6]$, a copy of~$S_6^*$ from Construction~\ref{CONSTRC:K43-} and place a copy of the~$3$-graph~$G$ from Construction~\ref{CONSTRC:bipartite-K43} between~$V_{3i-1}$ and~$V_{3i-2}$ for~$i =1,2$. 
Call the resulting graph~$H$.
We claim that~$H$ is~$K_{5}^{(3-)}$-free.

Let~$D$ be a set of five vertices of~$H$. 
If~$D$ intersects five of the~$V_i$, there are at least two edges missing on these five vertices. 
If~$D$ intersects some~$V_i$ in at least four vertices, then since the~$3$-graph in each~$V_i$ is~$K_4^{(3)-}$-free,~$D$ misses at least two edges.
Next, notice that for any~$i,j\in[6]$ with~$i\neq j$, the~$3$-graph given by those edges of~$H$ which intersect both~$V_i$ and~$V_j$ is~$K_4^{(3)}$-free.
Thus, if~$D$ has three vertices in some~$V_i$ and two vertices in some~$V_j$ with~$i\neq j$,~$D$ is missing at least two edges. 
There are three remaining cases:
\begin{itemize}
    \item For three distinct~$i,j,k\in[6]$, we have~$\vert D\cap V_i\vert=3$,~$D\cap V_j, D\cap V_k\neq\emptyset$.
    \item $D$ intersects one partition class~$V_i$ in exactly two vertices and three other distinct partition classes in exactly one vertex each.
    \item For three distinct~$i,j,k\in[6]$, we have~$\vert D\cap V_i\vert=\vert D\cap V_j\vert=2$, and~$D\cap V_k\neq\emptyset$.
\end{itemize}
In each of these cases, there are three distinct~$i,j,k\in[6]$ such that we can choose distinct~$v_1,v_2\in D\cap V_i$,~$v_3\in V_j$, and~$v_4\in V_k$.
If~$v_1v_2v_3$,~$v_1v_2v_4$ are both edges of~$H$, then by the construction of~$H$,~$v_1v_3v_4$~and~$v_2v_3v_4$ must be non-edges.
Hence, w.l.o.g., we may assume that~$v_1v_2v_4\notin E(H)$ and~$v_1v_2v_3\in E(H)$.
Note that then the last vertex~$v_5$ of~$D$ must lie in~$V_j$ or in~$V_{\ell}$ for some~$\ell\in[6]\setminus\{i,j,k\}$, otherwise there would be at least two non-edges on~$D$.
If~$v_5\in V_j$, then - as mentioned before - $H$ is missing one edge on~$\{v_1,v_2,v_3,v_5\}$ and so two on~$D$.
Now assume that~$v_5\in V_{\ell}$ for some~$\ell\in[6]\setminus\{i,j,k\}$.
Then again by the construction of~$H$, if both~$v_1v_2v_3$ and~$v_1v_2v_5$ are edges of~$H$, then~$v_1v_3v_5$ and~$v_2v_3v_5$ must be non-edges.
This yields two missing edges on~$D$.

In summary, we have shown that~$H$ is~$K_5^{(3)-}$-free.

The number of edges in~$H$ is 
\begin{align*}
    &(1 - o(1)) \Big(4x^3n^3 + 8x^2 y n^3+4x y^2 n^3 + 4yn\binom{xn}{2}+ 4 xn\binom{yn}{2}\\
    &+ 4 \cdot \frac{2}{7} \binom{xn}{3}+ 2\cdot \frac{2}{7} \binom{yn}{3}+ 2\cdot \frac{2}{3}(xn)^3 \Big) \\
    & =  (1-o(1))\left( \frac{232}{7} x^3 +\frac{4}{7} y^3 + 60x^2 y+36xy^2\right) \binom{n}{3}.
\end{align*}
Under the constraint~$4x+2y=1$, this function is maximized for~$x = \frac{45-\sqrt{277}}{184}$ and $y= \frac{1+\sqrt{277}}{92}$, yielding the desired bound~$\pi(K_5^{(3)-}) \geq \frac{31097 + 277 \sqrt{277}}{59248} = 0.602673\ldots$.
\end{proof}

\section{Concluding remarks}\label{SEC:CONCLUDING-Remarks}
In this paper, we develop a general approach to prove that Tur\'{a}n densities of certain hypergraphs are separated without using explicit bounds. 
Motivated by Erd\H{o}s's Question~\ref{CONJ:erdos-problem} and Theorem~\ref{THM:gapof-Kk+1-Kk+2-}, it would be interesting to study the following problem.

\begin{problem}\label{PRO:gap-Kl+2-}
    For~$\ell > k \geq 3$, is it true that~$\pi(K_{\ell}^{(k)}) < \pi(K_{\ell+1}^{(k)-})$ ?
\end{problem}

In fact, using a similar proof as for Theorem~\ref{THM:gapof-Kk+1-Kk+2-} for~$k \geq 4$, we can give a positive answer to this question for all~$k < \ell \leq 2k-3$ by showing that in this case Condition~\eqref{EQ:Statement-2} of Theorem~\ref{THM:general-result} is satisfied. 

Corollary~\ref{COR:gap-k+1vertex-t-t+2} motivates the following problem; see also Problem~$6.2$ in~\cite{DHLY24} for a related problem.
\begin{problem}\label{PRO:gap-k+1vertex-t-t+1}
    For~$3 \leq t \leq k$, is it true that~$\pi(H_{t}^{(k)}) < \pi(H_{t+1}^{(k)})$ ?
\end{problem}

\section*{Acknowledgements}
We thank David Conlon, Sim\'on Piga, and Marcelo Sales for helpful discussions.
We are further grateful to the Extremal Combinatorics and Probability Group at the Institute for Basic Science (IBS) for organizing the 2nd ECOPRO student research program. 
Research was supported by the Institute for Basic Science IBS-R029-C4 (Hong Liu), the Young Scientist Fellowship IBS-R029-Y7 (Bjarne Sch\"ulke), the China Scholarship Council and IBS-R029-C4 (Shuaichao Wang), Seed Fund Program for International Research Cooperation of Shandong University and IBS-R029-C4 (Haotian Yang), the National Key R\&D Program of China (Grant No. 2023YFA1010202) and IBS-R029-C4 (Yixiao Zhang).



\bibliographystyle{alpha}
\bibliography{separating}

\end{document}